\documentclass[10pt,article,reqno]{amsart}

\usepackage{amssymb}
\usepackage{amsthm}
\usepackage{lineno}
\usepackage{booktabs}
\usepackage{amsmath}
\allowdisplaybreaks[4]
\usepackage{subfigure}
\usepackage{epstopdf}
\usepackage[colorlinks=true,citecolor=blue]{hyperref}
\newtheorem{theorem}{Theorem}[section]
\newtheorem{lemma}{Lemma}[section]
\newtheorem{prop}{Proposition}[section]
\newtheorem{remark}{Remark}[section]
\numberwithin{equation}{section}

\begin{document}
	
	\title[]
	{A mean curvature flow approach to Hamilton's pinching theorem}

	\author{Liang Cheng, Zhenyu Lu$^{\ast}$ }
	
	\thanks{*Corresponding author} 
	
	\subjclass[2020]{Primary 53C24; Secondary 	53E20 .}

	\thanks{Liang Cheng's  Research partially supported by
		Natural Science Foundation of China 12171180
	}
	
	\address{School of Mathematics and Statistics, and Key Laboratory of Nonlinear Analysis $\&$ Applications (Ministry of Education), Central  China Normal University, Wuhan, 430079, P.R.China}

	\email{chengliang@ccnu.edu.cn }
	
	\address{School of Mathematics and Statistics, and Key Laboratory of Nonlinear Analysis $\&$ Applications (Ministry of Education), Central  China Normal University, Wuhan, 430079, P.R.China}
	
	\email{	quermassintegral@gmail.com}

	\begin{abstract}
		In this paper, we provide a proof of Hamilton's extrinsic pinching theorem using the mean curvature flow approach. 
	\end{abstract}

	\keywords{ 
		Hypersurface; Second fundamental form; Mean curvature flow; Extrinsic pinching theorem
	}
	
	\maketitle	
	\section{Introduction}\label{sec1}
	The Bonnet-Myers theorem is a classical result in Riemannian geometry. It states that a complete Riemannian manifold $(M^n, g)$ whose Ricci curvature has a positive lower bound $(n-1)k$ must be compact, and its diameter satisfies $\mathrm{diam}(M) \leq \pi / \sqrt{k}$. It is interesting to seek other curvature conditions that imply compactness for manifolds. In this direction, there exists a series of Bonnet-Myers type theorems which were obtained under pinched conditions, such as \cite{26,18,21,20,29,28,25,19}. All these pinched Bonnet-Myers type theorems were inspired by the following extrinsic pinching theorem for hypersurfaces, which was first obtained by Hamilton \cite{8}:
	\begin{theorem}\cite{8} \label{hamilton}
		Any smooth, proper, locally uniformly convex hypersurface $\operatorname{X}:M^n \to \mathbb{R}^{n+1}$, $n \geq 2$, with a pinched second fundamental form:
		$$\mathrm{II}\geq \alpha \operatorname{H}g \text{ for some } \alpha>0,$$ is necessarily compact. Here, $\mathrm{II}$ and $\operatorname{H}$ denote the second fundamental form and mean curvature for $\operatorname{X}$, ~g~is the induced Riemannian metric on~$M^n $.
	\end{theorem}

	Hamilton's idea of proof for Theorem \ref{hamilton} is based on the observation that the pinching condition implies that the Gauss map is a local conformal diffeomorphism, which in turn means that the conformal metric~$h$~on~$\mathcal{M}^n$~induces a complete metric of finite volume on $U\subset S^n$,~ultimately leading to a contradiction (see \cite{8} and \cite{Ni}).
	
	However, unlike Hamilton's original approach to proving the extrinsic pinching theorem, other Bonnet-Myers' type theorems under pinching conditions have been established via geometric flows. For example, studies on Bonnet-Myers type theorems under the Ricci-pinched condition $\operatorname{Ric} \ge \alpha R g>0$ on $3$-dimensional manifolds have been conducted using the Ricci flow \cite{18,20,29,19};
	other Bonnet-Myers type theorems under analogous pinching conditions for the higher-dimensional  manifolds  have also employed the Ricci flow \cite{26,LT,28,25};
	and Bonnet-Myers type theorems under the Ricci-pinched condition $\operatorname{Ric} \ge \alpha R g>0$ on $n$-dimensional locally conformally flat manifolds ($n \ge 3$) have utilized the Yamabe flow \cite{21,28}.

	
	It is natural to ask whether Hamilton's extrinsic pinching theorem  could be proved using methods from geometric flows, such as those used in the aforementioned pinched Bonnet-Myers type theorems. This question was addressed by Bourni, Langford, and Lynch \cite{14}, who provided such a proof under the additional assumption that the second fundamental form is bounded. In this paper, we give a proof of Hamilton’s extrinsic pinching theorem via mean curvature flow, without imposing any additional assumptions. Our approach is based on the following main theorem:
	\begin{theorem}\label{czx}
		For any given~$\alpha>0$ and $n>2$,~there exists a constant~$C\doteq C(\alpha,n)>0$~such that the following holds:~Let~$X_0:M^n\to \mathbb{R}^{n+1}$~be a complete non-compact hypersurface. If $M^n$ satisfies the pinching condition~$\mathrm{II} \geq \alpha  \operatorname{H}g \geq 0$,~then there exists a complete solution~$X(t)$ to the mean curvature flow for $t\in [0,+\infty)$~starting from $X_0$ such that
		
		(1)~$| \mathrm{II}(x, t)|^2 \leq \frac{C}{t},$
		
		(2)~${\mathrm{II}}(x, t) \geq \alpha  \operatorname{H}(x, t)g(x,t)\geq0,$\\
		for all~$(x, t) \in M^n \times(0, +\infty)$.   Here $C(\alpha,n)$~denotes a constant $C$ depending on $\alpha$ and $n$.
		
	\end{theorem}  
	
	\begin{remark}
		Theorem~\ref{czx}, in combination with the result of Bourni--Langford--Lynch (Theorem~7 in~\cite{14}), yields a proof of Hamilton's extrinsic pinching theorem via the mean curvature flow. Indeed, since we have already shown that any mean curvature flow satisfying the hypotheses of Theorem~\ref{czx} must be of Type~III, we may restrict the argument of Bourni--Langford--Lynch to the Type~III case, thereby simplifying their proof (see the proof of Theorem \ref{hamilton1}).
	\end{remark}

	Note that Ecker--Huisken proved the existence of mean curvature flow for complete noncompact hypersurfaces in n Euclidean space under a uniform local Lipschitz condition (see Theorem~4.2 in \cite{10}). However, in the absence of a uniform local Lipschitz condition, the general existence theory for mean curvature flow on noncompact hypersurfaces remains less developed.
	In order to prove Theorem~\ref{czx}, we need to construct a local mean curvature flow for a ball whose existence time depends solely on the pinching constant and the dimension $n$.
	Specifically, we need to establish the following result, from which Theorem~\ref{czx} follows by taking the limit $R \to \infty$: 
	\begin{theorem}\label{lcmcf}
		For any given $\alpha > 0$ and $n>2$, there exist constants $T(\alpha, n)$, $C(\alpha, n) > 0$ such that the following hold:
		Let $X_0: M^n \to \mathbb{R}^{n+1}$ be an embedded (but not necessarily complete) hypersurface. If the hypersurface $\mathcal{M}^n=X(M^n)$ satisfies the pinching condition
		$$ \mathrm{II}\left(x\right) \geq \alpha  \operatorname{H}\left(x\right)g\ge 0
		~~~in ~~~B_{R+4},$$
		then there exists a smooth mean curvature flow $X(t)$ defined on $B_{R} \times [0, T]$ with initial data $X(0) = X_0$, and satisfying:
		
		(i)~$| \mathrm{II}(x, t)|^2 \leq \frac{C}{t} ~~\text{on} ~~ B_{R} \times(0, T],$
		
		(ii)~$
		\mathrm{II}(x, t) \geq \left(\alpha  \operatorname{H}(x, t)-1\right)g(x,t)
		~~\text{on} ~~ B_{R} \times(0, T].$ \\
		Here, we have used the notation $  B_r\doteq \{x\in \mathbb{R}^{n+1} \big|\left|  x  \right|<r\}$, $  B(x_0,r)\doteq \{x\in \mathbb{R}^{n+1} \big|\left|  x  -x_0\right|<r\}$ and $  Q^{r}(p)\doteq B(X(p),r)\cap \mathcal{M}^n $ for $p\in M^n$ and $\mathcal{M}^n=X(M^n)$. 
	\end{theorem}

	Next, we sketch the strategy for the proof of Theorem~\ref{lcmcf}. Our argument is an iterative induction procedure inspired by the ideas of Lee--Topping in~\cite{29} and \cite{LT}.
	We begin by performing a conformal transformation on the initial metric restricted to a local ball. This transformation renders the metric complete and ensures that the second fundamental form is bounded, while leaving the metric unchanged on a smaller interior region. This construction is then applied inductively at each step of the iteration.
	In particular, we first choose an initial time $t_1 > 0$ and an initial radius $r_1 = R + 3$ such that the estimate
	\begin{equation}
		|\mathrm{II}(x, t)|^2 \leq \frac{C}{t}
		\label{eqQ}
	\end{equation}
	holds, and the pinching condition is almost preserved, on $Q^{r_1}(p) \times [0, t_1]$. 
	We then inductively define sequences of times $\{t_k\}$ and radii $\{r_k\}$ by
	\[
	t_{k+1} = (1 + C_1) t_k, \qquad r_{k+1} = r_k - C_2 t_k^{1/2},
	\]
	where $C_1, C_2 > 0$ are uniform constants independent of $k$.
	Assuming inductively that inequality~\eqref{eqQ} holds on $Q^{r_k}(p) \times [t_{k-1}, t_k]$, a crucial step in our proof is to use this hypothesis to show that the pinching condition remains almost preserved under the local mean curvature flow constructed above (see Lemma~\ref{weakpc}).  Based on this, we establish that the estimate~\eqref{eqQ} continues to hold on the next cylinder $Q^{r_{k+1}}(p) \times [t_k, t_{k+1}]$.
	By carrying out this inductive process, we ultimately obtain a local solution $X(t)$ to the mean curvature flow with the desired properties.

	The present paper is organized as follows. In Section \ref{section2} we review fundamental aspects of mean curvature flow. Furthermore, we will prove a number of auxiliary lemmas that will be essential for subsequent analysis.  Section \ref{section3} is devoted to priori estimates. Under certain prescribed conditions, we derive decay estimates for the second fundamental form and establish estimates confirming the preservation of the pinching condition over the entire mean curvature flow. Section~\ref{section4} is dedicated to the proof of the main Theorems~\ref{lcmcf} and~\ref{czx}. 
	Using the results and estimates established in Sections~\ref{section2} and~\ref{section3}, 
	we first complete the proof of Theorem~\ref{czx}, and then provide a proof of Hamilton's extrinsic pinching theorem via the mean curvature flow approach.   
	
	\section{Preliminaries}\label{section2}
	A family of smooth embeddings ~$X(t):M^n\times [0,T)\to \mathbb{R}^{n+1}$ is said to evolve by mean curvature flow if
	\begin{equation}
		\begin{split}
			\left\{ \begin{array}{l}
				\partial _tX\left( x,t \right) =-\operatorname{H}\left( x,t \right) N\left( x,t \right) ,~~~x\in M^n,t>0,\\
				X\left( \cdot ,0 \right) =X_0( \cdot ),~~~~~~~~~~~~~~~~~~~~~~~~~x\in M^n,
			\end{array} \right.
		\end{split}
		\label{eqmcf}
	\end{equation}
	where ~$N\left( x,t \right)$~is the unit outward normal vector of the hypersurface~$\mathcal{M}^n_t=X(M^n,t)$ and~$\operatorname{H}\left( x,t \right)$~is the mean curvature.
	
	Along the flow~\eqref{eqmcf}, we have the following evolution equations on the induced metric $g_{ij}$, unit outward normal $N$, the second fundamental form $\mathrm{II}$ and mean curvature $H$ of $\mathcal{M}_t$:  
	\begin{equation}
		\partial _tg_{ij}=-2\operatorname{H}\mathrm{II}_{ij},
		\label{g}
	\end{equation}
	\begin{equation}
		\partial _tN=\nabla \operatorname{H},
		\label{N}
	\end{equation}
	\begin{equation}
		\partial _t\mathrm{II}_{ij}=\varDelta \mathrm{II}_{ij}+\left| \mathrm{II} \right|^2\mathrm{II}_{ij}-2H\mathrm{II}^2_{ij},
		\label{2}
	\end{equation}
	\begin{equation}
		\partial _t\operatorname{H}=\varDelta \operatorname{H}+\left| \mathrm{II} \right|^2\operatorname{H},
		\label{h}
	\end{equation}
	\begin{equation}\label{sec1:evolution|x|}
		\left( \partial _t-\varDelta \right) \left| X \right|^2=-2n.
	\end{equation}
	Here, $\nabla $~denotes the  Levi-Civita connection with respect to the induced metric $g_{ij}$ on $X_t$.

	\subsection{Formulas for Conformal Changes}
	
	We first study how the second fundamental form changes under a conformal transformation.

	\begin{prop}\label{gx}
		Let $X \colon M^n \to (\mathbb{R}^{n+1}, g)$ be an immersion with second fundamental form $\mathrm{II}$. For any smooth function $f \in C^\infty(\mathbb{R}^{n+1})$, define the conformally related metric $\tilde{g} = e^{2f}g$. Then the norm of the second fundamental form $\widetilde{\mathrm{II}}$ of the immersion $X \colon M^n \to (\mathbb{R}^{n+1}, \tilde{g})$ are given by:
		$$
		\left|\widetilde{\mathrm{II}}\right|_{\tilde{g}} = e^{-f}\left|(\mathrm{II} - g\nabla_N f)\right|_g,
		$$
		where $\nabla_N$ denotes the covariant derivative with respect to the Levi-Civita connection of $g$ in the normal direction.
	\end{prop}
	\begin{proof}
		Let $\{x_1,\cdots,x_{n+1}\}$ be local coordinates such that $\{x_1,\cdots,x_{n}\}$ air  local coordinates for $M^n$ and $\frac{\partial }{\partial x_{n+1}} =N$ .
		Since the ~Christoffel~ symbol satisfies
		\begin{equation*}
			\begin{split}
				\widetilde{\varGamma } _{ij}^{k}&=\frac{1}{2}\tilde{g}^{kl}\left( \partial _j\tilde{g}_{li}+\partial _i\tilde{g}_{lj}-\partial _l\tilde{g}_{ij} \right) \\
				&=\frac{1}{2}e^{-2f}g^{kl}\left( \partial _j\left( e^{2f}g_{li} \right) +\partial _i\left( e^{2f}g_{lj} \right) -\partial _l\left( e^{2f}g_{ij} \right) \right) \\
				&=\frac{1}{2}e^{-2f}g^{kl}\left( 2e^{2f}\nabla _jfg_{li}+e^{2f}\partial _jg_{li}+2e^{2f}\nabla _ifg_{lj}+e^{2f}\partial _ig_{lj}-2e^{2f}\nabla _lfg_{ji}-e^{2f}\partial _lg_{ji} \right) \\
				&=\varGamma _{ij}^{k}+\delta_{ki}\nabla _jf+\delta_{kj}\nabla _if-g_{ij}\nabla _kf,\\
			\end{split}
			\label{eq:label}
		\end{equation*}
		we have
		\begin{equation*}
			\begin{split}
				\widetilde{\mathrm{II}}_{ij}&=\tilde{g}\left< \tilde{\nabla}_{\partial _i}\partial _j,\widetilde{N} \right> \\
				&=\tilde{g}\left< \widetilde{\varGamma }_{ij}^{k}\partial _k,e^{-f}\partial _{n+1} \right>\\ 
				&=e^{f}\widetilde{\varGamma }_{ij}^{n+1}\\
				&=e^{f}(\varGamma _{ij}^{n+1}+\delta_{n+1,i}\nabla _jf+\delta_{n+1,j}\nabla _if-g_{ij}\nabla _{n+1}f)\\
				&=e^{f}(\mathrm{II}_{ij}-g_{ij}\nabla _{n+1}f).\\
			\end{split}
			\label{eq:label}
		\end{equation*}
		This implies $$
		\left|\widetilde{\mathrm{II}}\right|_{\tilde{g}} = e^{-f}\left|(\mathrm{II} - g\nabla_N f)\right|_g.
		$$
	\end{proof}

	To construct a local mean curvature flow, we need to modify a small incomplete region on the hypersurface. The goal is to make the metric on this local patch complete, while leaving the interior unchanged and without altering its second fundamental form too much. This necessitates the construction of a cut-off function using the distance function.
	Given an open set $U$ on the hypersurface, performing a conformal transformation on the metric over $U$ to "push the boundary to infinity" constitutes a crucial step in the completion of this local region.
	\begin{lemma}\label{wbh}
		Let~$X:M^n\to (\mathbb{R}^{n+1},g_{edu})$~be an embedding with the second fundamental form $\mathrm{II}$ and induced metric $g$. Let $U \subset \mathcal{M}^n $ be an open domain satisfying $|\mathrm{II}(x)|\leq 1$~on $ U$ and $Q^{1}(p)\subset \subset U$ for a point $X(p) \in U$. Then there exist an open domain ~$\tilde{U}$~ such that ~$(U)_1\subset \tilde{U} \subset U$~ and for any $l> C(n)$ one can produce an induced metric ~$\tilde {h}_l$~ on ~$\tilde{U}$ with
		
		(1)~$(\tilde{U},\tilde {h}_l)$~is complete,
		
		(2)~$\tilde {h}_l \equiv g $ on $ (\tilde{U})_{\sqrt{\frac{C\left( n \right)}{l}}},$
		
		(3)~$|\widetilde{\mathrm{II}}(x)|_{\tilde {h}_l}\leq l ~~\text{on}~~~ \tilde{U}.$\\ 
		Here,
		we denote $\left( U \right) _r\doteq\left\{ x\in U\ |\ B(x,r)\cap \mathcal{M}^n  \subset \subset U \right\} $. 
	\end{lemma}
	\begin{proof}
		Without loss of generality, we can select an open set ~$V=B_2 \subset \mathbb{R}^{n+1}$. 
		Consider the real valued function defined on $V$ such that
		$$x\to \text{max}(0, 1-2d_{g_{edu}}(x,\mathbb{R}^{n+1}\backslash B_{\frac{3}{2}}) ) .$$
		We smooth the real valued function, obtaining a function ~$\rho\in C^\infty(\mathbb{R}^{n+1})$ such that
		\begin{equation*}
			\begin{split}
				&\rho(x) \equiv 0, ~~ x\in B_1,\\
				&\rho (x) \equiv 1, ~~x\in  \partial V,\\ 
				&\left| \nabla \rho \right|~,~\left| \nabla ^2\rho \right|\le C\left( n \right). 
			\end{split}
		\end{equation*}\
		This justifies the choice of $\tilde{V}=\{x\in \mathbb{R}^{n+1}| \rho(x) <1\}.$
		
		Next, we fix ~$0<\varepsilon <1$, where a suitable $\varepsilon$ will be determined later, and define the conformal metric ~$h=e^{2f\circ\rho }g_{edu}$, where $f:[0,1]\to \mathbb{R}_+$ is a $C^\infty$ function obtained by smoothing the following function in a $\delta$-neighborhood of $1-\varepsilon$
		$$  f\left( x \right) =\left\{ \begin{matrix}
			0&      \ \ \ \ 0 \leq x\le 1-\varepsilon ,\\
			-\ln \left( 1-\left( \frac{x-1+\varepsilon}{\varepsilon} \right) ^2 \right)&        \ \ \ 1-\varepsilon \le x<1.\\
		\end{matrix} \right. $$
		Then $f$  satisfies  $f(x)=0$ for~$0 \leq x\leq 1-\varepsilon -\delta$,~$f\left( x \right) =-\ln \left( 1-\left( \frac{x-1+\varepsilon}{\varepsilon} \right) ^2 \right)$~for~$1>x\ge 1-\varepsilon +\delta$, $0<f'\left( x \right) <\frac{2\varepsilon}{\varepsilon ^2-\left( x-1+\varepsilon \right) ^2}$~and~$0<f''\left( x \right) <\frac{4\varepsilon ^2}{\left( \varepsilon ^2-\left( x-1+\varepsilon \right) ^2 \right) ^2}$~for~$1-\varepsilon +\delta < x <1$.
		
		Choose a sequence of points $\{ x_k \}$ such that
		~$$\left\{ x_k\in \tilde{V}\ |\ d_{g_{edu}}\left( x_k,~\mathbb{R}^{n+1} \backslash \tilde{V} \right) = \frac{1}{k} \right\},$$
		where $d_{g_{edu}}$ is the distance function under the Euclidean metric $g_{edu}$. Then $x_k\to \partial \tilde{V}$~for~$k\to \infty .$ Under the choice of a conformal metric~$h$ ~ in~$\tilde{V}$, the distance between ${x_k}$ and the origin $o$ tends to infinity:
		$$d_{h}\left( x_k,o \right) =\left| x_k \right|_{h}=e^{2f\circ \rho}\left| x_k \right|_{g_{edu}}\triangleq P_k\rightarrow +\infty.$$
		Let ~$K_i=\overline{B_{h}\left( o,P_i \right) } \subset \tilde{V}$. Then $K_i\subset K_{i+1}$~for all i. If ~$q_j\notin K_j$,~we have$$d_{h}\left( o,q_j \right) \ge d_{h}\left( o,x_j \right) \rightarrow +\infty.$$
		Therefore $(\tilde{V},h)$~is complete. 
		
		Now we discuss the completion of open set ~$U$~ in hypersurfaces $\mathcal{M}^n$. For any open set~$U\subset \subset \mathcal{M}^n$, there exists an open set~$V \subset \mathbb{R}^{n+1}$ ~ such that~$U\triangleq V\cap \mathcal{M}^n$.~ By conformal transformation for ~$V$~ and denoting ~$\tilde{U}\triangleq \tilde{V}\cap \mathcal{M}^n \subset U$, for a constant $l> C(n)$ we can get an induced metric in~$\tilde{U}$ ~
		$${(\tilde{h}_l)}_{ij}=h_{l}\left< \partial _iX,\partial _jX \right>=e^{2f\circ \rho}g_{edu}\left< \partial _iX,\partial _jX \right>=e^{2f\circ \rho}g_{ij}.$$
		Similarly, we choose a sequence of points $\{x_k\}\in \mathcal{M}^n$~satisfying~$$\left\{ x_k\in \tilde{U}\ |\ d_{g}\left( x_k,~\mathcal{M}^n \backslash \tilde{U} \right) = \frac{1}{k} \right\}.$$
		Then~$\tilde{P_k}\triangleq  |x_k|_{\tilde{h}_l} \rightarrow +\infty$. This implies $(\tilde{U},\tilde{h}_l)$ is complete.
		
		Moreover,~by Proposition~\ref{gx}~$$
		\widetilde{\mathrm{II}}_{ij} = e^{f} (\mathrm{II}_{ij} - \nabla _Nf \cdot g_{ij}) ,$$
		we have
		\begin{equation*}
			\begin{split}
				|\widetilde{\mathrm{II}}|_{\tilde {h}_l}&=e^{-f}|\left( \mathrm{II}-\nabla _Nf\cdot g \right) |_{g} \\
				&\le e^{-f}(\left| \mathrm{II} \right|_{g}+\left| \nabla _Nf\cdot g \right|_{g}) \\
				&\le e^{-f}(\left| \mathrm{II} \right|_{g}+\left| f'\right|\cdot\left| \nabla \rho \right|_{g}\cdot \left| g \right|_{g}) \\
				&\le \frac{C(n)}{\varepsilon }. \\
			\end{split}
			\label{eq:label}
		\end{equation*}
		Fix~$\varepsilon = \frac{C\left( n \right)}{l}
		$~. Therefore~$|\widetilde{\mathrm{II}}(x)|_{\tilde {h}_l}\leq l.$
	\end{proof}
	
	\subsection{Some results on mean curvature flow}
	We now recall some results on mean curvature flow that will be needed later. First, we require the following short-time existence theorem for the mean curvature flow of a complete, non-compact hypersurface in a Riemannian manifold with bounded curvature, assuming that the initial second fundamental form is bounded.
	
	\begin{theorem} \cite{9} \label{zhu}
		Let~$M^n$~be compact or let the curvature operator of~$N^{n+1}$~be bounded. Let~$X_0: M^n \to N^{n+1}$~be a smooth, orientable immersion. Suppose that there exists a constant~$ \rho>0$~such that the second fundamental form of
		~$X_0$~satisfies~$| \mathrm{II}|^2 \leq  \rho^{-2}$. Then there exists a smooth solution~$X(\cdot,t)$ ~ defined on~$M^n\times [0,\gamma(n) \rho^2)$~, which satisfies
		$$
		\left\{ \begin{array}{l}
			\partial _tX\left( x,t \right) =-\operatorname{H}\left( x,t \right) N\left( x,t \right) ,~~x\in M^n,t>0,\\
			X\left( \cdot ,0 \right) =X_0,~~~~~~~~~~~~~~~~~~~~~~~~~~~~x\in M^n,\\
		\end{array} \right. 
		$$
		and$$| \mathrm{II}(t)|^2 \leq  \lambda(n)\rho^{-2},$$
		where~$\gamma(n)$~and~$ \lambda(n)$~are positive constants depending only on the dimension $n$.
	\end{theorem}

	Second, we need the following localized version of Huisken’s umbilic estimate due to Langford~\cite[Theorem~1]{31} with $m=0$ (also see Proposition~1 in \cite{14}).
	
	\begin{lemma} \label{umblic} \cite{31}   
		Assume that the mean curvature flow~$X(t):M^n\to \mathbb{R}^{n+1}$~is proper and $\mathrm{II}(x,t)\geq \alpha \operatorname{H}(x,t)g(x,t)$~holds for all~$(x,t)\in B_{2L}\times(0,T]$.~Then
		$$
		|\mathrm{II}-\frac{1}{n}\operatorname{H}g|\le \varepsilon \operatorname{H}+C_{\varepsilon}\varTheta, 
		~(x,t)\in B_{\frac{L}{2}}\times(0,T],$$  
		where~$
		\varTheta =\sup\limits_{B_{2L} \times \left\{ 0 \right\} \bigcup{B_{2L}\backslash B_{L} \times \left( 0,T \right)}}\operatorname{H}
		$ and $
		C_{\varepsilon}=C\left( n,\alpha ,\varepsilon \right) 
		$.
	\end{lemma} 
	Finally, the following proposition asserts that there exist no locally uniformly convex, pinched mean curvature flow translators or expanders.
	\begin{prop}\label{no_existence_te} \cite{14}  
		There exist no non-flat mean curvature flow translators or expanders for which
		\[
		\mathrm{II}(x) \geq \alpha \operatorname{H}(x)\, g \geq 0
		\]
		holds with some constant $\alpha > 0$.
	\end{prop}
	
	\subsection{ Point-picking lemma}
	The following lemma provides a geometric containment estimate for mean curvature flow. It will be used as a key ingredient in the proof of Lemma~\ref{pp} below.
	
	\begin{lemma}\label{ssq}
		Suppose that the mean curvature flow ~$X(t):M^n\to \mathbb{R}^{n+1}$~satisfies $ | \mathrm{II}(x,t) |^2\leq f(t)$ in ~$  Q^{r,0}(x_0)\times (0,T]$. Then
		\begin{equation}
			Q^{r-2\sqrt{n}\int_0^t{\sqrt{f(t)}dt}, t}(x_0)\subset Q^{r,0}(x_0),
		\end{equation}
		\textbf{Here and below, we use the notation~$  Q^{r,t}(p)\doteq B(X_t(p),r)\cap \mathcal{M}^n_t $ for any $p\in M^n.$  }	
	\end{lemma}
	\begin{proof}
		Let $Y_t(x)=X_t(x)-X_t(x_0)$.
		Since \begin{equation*}
			\begin{split}
				\partial _t\left| Y_t \right|^2&=2\left< \partial _tY_t,Y_t \right> \\
				&=-2\left<\operatorname{H}(x,t)N(x,t)-\operatorname{H}(x_0,t)N(x_0,t) ,Y_t \right> \\
				&\ge -2\left|Y_t\right|(\left|\operatorname{H}(x,t)\right|+\left|\operatorname{H}(x_0,t)\right|)\\
				&\ge -4\sqrt{nf(t)}\left|Y_t\right|.\\ 
			\end{split}
			\label{eq:label}
		\end{equation*}
		
		Integrating both sides yields
		\begin{equation*}
			\begin{split}
				\int_0^t{\partial _t(2|Y_t|)dt} &\ge -4\sqrt{n}\int_0^t{\sqrt{f(t)}dt},\\
				\left( |Y_t|-|Y_0| \right) &\ge -2\sqrt{n}\int_0^t{\sqrt{f(t)}dt},\\
				|Y_0|&\le |Y_t|+2\sqrt{n}\int_0^t{\sqrt{f(t)}dt}.
			\end{split}
		\end{equation*}
		That is
		\begin{equation*}
			Q^{r-2\sqrt{n}\int_0^t{\sqrt{f(t)}dt}, t}(x_0)\subset Q^{r,0}(x_0).
		\end{equation*}
		
	\end{proof}
	
	The following lemma asserts that, under suitable conditions, a mean curvature flow either exhibits time decay of its second fundamental form or, equivalently, admits local control of the norm of the second fundamental form. The proof is based primarily on Perelman’s point-picking technique.
	
	\begin{lemma}\label{pp}
		For any ~$n>2$~ ~$(\beta =4\sqrt{n})$, $a >0$ and mean curvature flow ~$X(t):M^n\to \mathbb{R}^{n+1}$ ~$(t\in [0,T]
		)$~, if ~$x_0\in M^n$~ satisfies ~$Q^{1, 0}\left(x_0 \right) \subset \subset \mathcal{M}^n$, then at least one of the following assertions holds:
		
		(1)~For~$t\in (0,T]$ with $t<\frac{1}{\beta ^2a}$,~it holds that~$Q^{1-4\sqrt{nat}, t}\left( x_0 \right) \subset Q^{1,0}\left( x_0\right) $~and$$~~~| \mathrm{II} \left(x,t\right)  |^2 < at^{-1}~~~throughout ~~ Q^{1-4\sqrt{nat},t}\left( x_0 \right).$$
		
		(2)~There exists~$\tilde{t}\in \left( \left. 0,T \right] \right.$~for~$\tilde{t}<\frac{1}{\beta ^2a}$~and~$X_{\tilde{t}}(\tilde{x})\in Q^{1-4\sqrt{na\tilde t},{\tilde t}  }\left( x_0 \right) $~such that~$$A^2=\left| \mathrm{II}\left(\tilde x, \tilde t \right) \right|^2\geq a\tilde{t}^{-1},$$and$$
		\left| \mathrm{II}\left( x,t \right)  \right|^2\le 4A^2=4\left| \mathrm{II}\left( \tilde{x},\tilde{t} \right)  \right|^2
		,$$
		whenever~$X_{ \tilde{t} } (x)\in Q^{\frac{\beta}{8}aA^{-1} ,{ \tilde{t} }  }\left( \tilde{x} \right) $,~$\tilde{t}-\frac{1}{8}aA^{-2}\le t\le \tilde{t}
		.$
	\end{lemma}
	\begin{proof}
		
		First, we claim the following.
		
		\textbf{Assertion (3):} There exists $\tilde{t} \in (0, T]$ with $\tilde{t} < \frac{1}{\beta^2 a}$ and $X_{\tilde{t}}(\tilde{x}) \in Q^{1 - 4\sqrt{na\tilde{t}},{\tilde{t}}  }\left(x_0 \right)$ such that
		\[
		\left| \mathrm{II}(\tilde{x}, \tilde{t}) \right|^2 \geq a \tilde{t}^{-1},
		\]
		and
		\[
		\tilde{r} = \left| X_{{\tilde{t}}}(\tilde{x}) - X_{{\tilde{t}}}(x_0) \right| + \frac{\beta a}{4} \left| \mathrm{II}(\tilde{x}, \tilde{t}) \right|^{-1} \leq 1 - \frac{1}{4}\beta \sqrt{a\tilde{t}} \leq 1,
		\]
		but
		\[
		\left| \mathrm{II}(x, t) \right|^2 \leq 4 \left| \mathrm{II}(\tilde{x}, \tilde{t}) \right|^2
		\]
		for all $t \in (0, \tilde{t}]$, $X_t(x) \in Q^{\tilde{r},{t}}(x_0)$, and $\left| \mathrm{II}(x, t) \right|^2 \geq a t^{-1}$.\\
		Next, we will prove that $X(t)$  satisfies either Assertion (1) or Assertion (3).

		Due to the smoothness of the mean curvature flow, we can choose a sufficiently small $t > 0$ such that 
		\[
		X_t(x) \in Q^{1 - 4\sqrt{nat},t }\left(x_0 \right) \subset \subset \mathcal{M}_t
		\]
		and 
		\[
		\left| \mathrm{II}(x, t) \right|^2 < a t^{-1}.
		\]
		In addition by Lemma~\ref{ssq},~we have~$Q^{1 - 4\sqrt{nat},t }\left(x_0 \right)\subset Q^{1,0}(x_0)$.~ If ~$X(t)$ doesn't satisfy Assertion (1), then there exists the first time ~$t_1\in (0,T]$~for~$t_1<\frac{1}{\beta ^2a}$~and point~$X_{t_1}(x_1)\in \overline{Q^{1-4\sqrt{nat_1},{ t_1}}\left(x_0 \right) }$~such that  ~$$\left| \mathrm{II}\left( x_1,t_1 \right)  \right|^2= at_1^{-1}.$$
		If ~$(x_1,t_1)$~ is the point required by Assertion (3), then Lemma $\ref{pp}$ is immediately established. Otherwise, there exists $(x_2,t_2)\in M\times(0,t_1]$ such that $$\left| \mathrm{II}\left( x_2,t_2 \right) \right|^2\ge \frac{a}{t_2},
		$$~
		$$
		\left| \mathrm{II}\left( x_2,t_2 \right) \right|^2\ge 4\left| \mathrm{II}\left( x_1,t_1 \right) \right|^2,
		$$
		and also satisfying
		\begin{equation}
			\begin{split}
				\left| X_{t_2}(x_2)-X_{t_2}(x_0) \right|&<\left| X_{t_1}(x_1)-X_{t_1}(x_0) \right|+\frac{\beta a}{4}\left| \mathrm{II}\left( x_1,t_1 \right) \right|^{-1}\\
				&\le \left( 1-\beta \sqrt{at_1} \right) +\frac{1}{4}\beta \sqrt{at_1}\\
				&=1-\left( 1-\frac{1}{4} \right) \beta \sqrt{at_1}\\
				&\le 1-\frac{3}{4}\beta \sqrt{at_2}.
			\end{split}
			\label{eq:label}
		\end{equation}

		Similarly,~if ~$(x_2,t_2)$~is the one required in Assertion~(3),~or if not, we can find $(x_3,t_3)\in M\times(0,t_2]$ satisfying
		~$$
		\left| \mathrm{II}\left( x_3,t_3 \right) \right|^2\ge \frac{a}{t_2},
		$$
		$$
		\left| \mathrm{II}\left( x_3,t_3 \right) \right|^2\ge 4\left| \mathrm{II}\left( x_2,t_2 \right) \right|^2,
		$$
		and
		\begin{equation}
			\begin{split}
				\left| X_{t_3}(x_3)-X_{t_3}(x_0) \right|&<\left| X_{t_2}(x_2)-X_{t_2}(x_0) \right|+\frac{\beta a}{4}\left| \mathrm{II}\left( x_2,t_2 \right) \right|^{-1}\\
				&\le \left| X_{t_1}(x_1)-X_{t_1}(x_0) \right|+\frac{\beta a}{4}\left| \mathrm{II}\left( x_1,t_1 \right) \right|^{-1}+\frac{1}{2}\frac{\beta a}{4}\left| \mathrm{II}\left( x_1,t_1 \right) \right|^{-1}\\
				&=1-( 1-\frac{1}{4} -\frac{1}{8}  ) \beta \sqrt{at_1}\\
				&\le 1-\frac{5}{8}\beta \sqrt{at_3}.
			\end{split}
			\label{eq:label}
		\end{equation}
		
		By iterating this procedure, we obtain $t_k\in (0,T]$~for~$t_k<\frac{1}{\beta ^2a}$~and point~$x_k\in M$~satisfying~$$
		\left| \mathrm{II}\left( x_k,t_k \right) \right|^2\ge \frac{a}{t_k},
		$$
		$$
		\left| \mathrm{II}\left( x_k,t_k \right) \right|^2\ge 4^{k-1}\left| \mathrm{II}\left( x_1,t_1 \right) \right|^2,
		$$~
		and
		\begin{equation}
			\begin{split}
				\left| X_{t_k}(x_k)-X_{t_k}(x_0) \right|&<1-(1-\frac{1}{4}-\frac{1}{8}-...-\frac{1}{2^k})\beta \sqrt{at_1}\\
				&\le1- \frac{1}{2}\beta \sqrt{at_k}
			\end{split}
			\label{eq:label}
		\end{equation}
		for large enough k.
		
		The uniform boundedness of the second fundamental form in ~$
		Q^{1,t}\left(x_0 \right) \times \left[ 0,T \right] 
		$~ ensures that finite iterations of this process will necessarily yield a point fulfilling Assertion (3).
		
		We now prove that Assertion (2) follows from Assertion (3). Indeed, the point $(\tilde{x},\tilde{t})$ satisfying Assertion (3) is precisely the point where Assertion (2) holds.
		Since ~$
		A^2=\left| \mathrm{II}\left( \tilde{x},\tilde{t} \right) \right|^2\ge \frac{a}{\tilde{t}}
		$ . Then for ~$t\in \left[ \tilde{t}-\frac{1}{8}aA^{-2},~\tilde{t} \right] 
		$,~it follows that~$
		\frac{7}{8}\tilde{t}\le t
		$.
		Hence, if Assertion (2) fails to hold at certain points ~$(x,t)$~, it follows necessarily that$$
		\left| \mathrm{II}\left( x,t \right) \right|^2>4A^2\ge \frac{4a}{\tilde{t}}\ge \frac{7a}{2t}\ge \frac{a}{t}
		.$$
		This contradicts Assertion (3).
		
		It remains to show the inclusion ~$
		Q^{\frac{a\beta}{8}A^{-1},{\tilde{t}}}\left(\tilde{x} \right) \subset Q^{\tilde{r},t}\left(x_0\right) 
		$. This follows by applying Lemma~\ref{ssq} to the curvature estimates derived from Assertion (3) in ~$Q^{\tilde{r},t}\left(x_0\right)$~.
		By applying Lemma~\ref{ssq},~we get$$
		Q^{\tilde{r}-\beta A\left( \tilde{t}-t \right),{\tilde{t}}}\left(x_0 \right) \subset Q^{\tilde{r},t}\left(x_0\right) 
		.$$
		Since~$
		\frac{7}{8}\tilde{t}\le t
		$,~we conclude $$
		\tilde{r}-\beta A\left( \tilde{t}-t \right) \ge \tilde{r}-\frac{1}{8}\beta aA^{-1}=\left| X_{\tilde{t}}\left( \tilde{x} \right) -X_{\tilde{t}}\left( x_0 \right) \right|+\frac{a\beta}{8}A^{-1}
		.$$
		This shows that$$
		Q^{\frac{a\beta}{8}A^{-1},{\tilde{t}}}\left(\tilde{x} \right) \subset Q^{\left| X_{\tilde{t}}\left( \tilde{x} \right) -X_{\tilde{t}}\left( x_0 \right) \right|+\frac{a\beta}{8}A^{-1},{\tilde{t}}}\left(x_0 \right) \subset Q^{\tilde{r},t}\left(x_0 \right) 
		,$$
		as required.
	\end{proof}
	
	\section{A Priori Estimates}\label{section3}
	\subsection{ Preserving pinching}
	First, we prove that the pinching condition $\mathrm{II}(x,0) \geq \alpha \operatorname{H}(x,0) g(x,0) > 0$ is almost preserved under a local mean curvature flow, provided that the growth condition $|\mathrm{II}(x,t)|^2 \leq a t^{-1}$ is satisfied along the flow.
	\begin{lemma}\label{weakpc}
		For mean curvature flow~$X(t):M^n\to \mathbb{R}^{n+1}$ with $Q^{1,0}(x_0)\subset\subset \mathcal{M}^n_t=X_t(M^n)$~$(t\in[0,T))$, if there exist $a>0$ and $\alpha>0$~such that:
		
		(1)~~$\mathrm{II}(x,0)\geq \alpha \operatorname{H}(x,0)g(x,0)>0 ~~~\text{in}~~ Q^{1,0}(x_0),$
		
		(2)~~$|\mathrm{II}(x,t)|^2\leq at^{-1}~~~\text{in}~~~ Q^{1,0}(x_0) \times  (0,T],$
		
		\noindent  then there exists a constant~$ S(a,\alpha)>0$~such that \begin{eqnarray*} \mathrm{II}(x_0,t)\geq (\alpha \operatorname{H}(x_0,t)-1)g(x_0,t), \end{eqnarray*} for ~$t\in [0,T\land S],$ where $T\land S \doteq \text{min} ~\{T,S \} $. 
		
	\end{lemma}
	
	\begin{proof}
		
		For all~$(x,t)\in M^n\times [0,T)$,~we define
		$$\lambda (x,t)=~\operatorname{inf}\{s\ge 0 :k_1(x,t)-\alpha \operatorname{H}(x,t)+s>0\},$$
		where $k_1 \le ... \le k_n$ are the principal curvatures of the evolving hypersurface $\mathcal{M}^n_t$. 
		Note that:\\ 
		(a)~$\lambda (x,0) =0~~\text{in}~~ Q^{1,0}(x_0),$\\
		(b)~$\lambda (x,t)\leq C\sqrt{at^{-1}}~~~\text{in}~~ Q^{1,0}(x_0)\times(0,T),$~and by ~\eqref{2}~and~\eqref{h}, we obtain
		\begin{equation*}
			\left( \partial _t-\varDelta \right) \left( \mathrm{II}-\alpha \operatorname{H}g \right) \\
			=\left| \mathrm{II} \right|^2\left( \mathrm{II}-\alpha \operatorname{H}g \right).\\
		\end{equation*}
		This shows that 
		\begin{equation}\label{barrier}
			( \partial _t-\varDelta )  \lambda \\
			\leq L\lambda \\ ~~~~~~~\text{in~the~barrier~sense},
		\end{equation}
		where $C= C(n,\alpha)$ is a constant and $L : Q^{1,0}(x_0)\times \left[ 0,T \right] \to R $ ~is continuous with~$L(x,t) \le at^{-1}$.~
		
		We now apply the local maximum principle. From \eqref{sec1:evolution|x|}, we obtain
		$$ ( \partial _t-\varDelta )(\left| X_t \right|^2+2nt)=0
		.$$
		Define~$\phi :[0,\infty)\to [0,1]$~by
		$$\phi \left( s \right) =\left\{ \begin{array}{l}
			\begin{matrix}
				1&      \ \ \ \ \ \ \ \ \ \ \ \ \ \ \ \ 0\le s\le \frac{1}{2},\\
			\end{matrix}\\
			\begin{matrix}
				\exp \left( -\frac{1}{1-s} \right)&     \\
			\end{matrix}\ \frac{3}{4}\le s\le 1,\\
			\begin{matrix}
				0&      \ \ \ \ \ \ \ \ \ \ \ \ \ \ \ \ \ \ \ \ s\ge 1,\\
			\end{matrix}\\
		\end{array} \right. $$
		(smoothly interpolated on $(\frac{1}{2},\frac{3}{4})$). There exists~$c>0$ with ~$\phi '\le 0$,~$\phi ''\ge -c\phi$. The cutoff function ~$\varPhi ( x,t ) =e^{-ct}\phi ({| X_t(x) |^2+2nt} )$ satisfies \begin{eqnarray*}
			\left( \partial _t-\bigtriangleup \right) \varPhi \le 0~~~\text{on} ~M^n\times[0,T].
		\end{eqnarray*}
		We define the function~$$F(x,t)=-\varPhi{^m}(x,t)\lambda(x,t)+\eta(t),$$where $m$ is a positive integer to be determined later and
		$\eta$ is chosen such that~$F>0$ near~$t=0$. If $F(x,t)<0$ at some point in~$ Q^{1,0}(x_0)\times [0,T]$, there must exist a point~$(X_{t_0}(p_0),t_0)\in Q^{1,0}(x_0)\times (0,T]$~satisfying:\\ 
		(1)~$F(p_0,t_0)=0$,\\
		(2)~~$F(x,t)>0$ in $ Q^{1,0}(x_0)\times [0,t_0).$\\
		
		From \eqref{barrier},~for any~$\varepsilon >0$,~there exists a~$C^2$~function~$\zeta(x)$~defined in a neighborhood of~$p_0$~such that:\\ 
		(1)~$\zeta\leq \lambda(x,t_0)$,~\\
		(2)~$\zeta(p_0)=\lambda (p_0,t_0)$,~\\
		(3)~the following inequality holds at ~$p_0$~
		$$ \frac{\partial _-}{\partial t}\lambda \left( p_0,t_0 \right) -\varDelta \zeta \left( p_0 \right) -L(p_0,t_0) \zeta \left( p_0 \right) \le \varepsilon.$$
		
		Let ~$G(x,t)=-\varPhi^m(x,t)\zeta(x)+\eta(t)$. The function $G$~is $C^2$~and satisfies ~$G(p_0,t_0)=F(p_0,t_0)=0$. Since ~$\lambda (x,t_0)>0$~and ~$\lambda$~is continuous,   for all $x$ sufficiently close to $p_0$, we obtain the following inequality:
		$$G(x,t_0)\ge -\varPhi{^m}(x,t_0)\lambda(x,t_0)+\eta(t_0)=F(x,t_0)\ge0.$$
		
		Because $(p_0,t_0)$~is the minimum point of $G$,~ the following conditions holds:  $$\left\{ \begin{array}{l}
			\zeta =\frac{\eta}{\varPhi ^{m}},\\
			\varPhi\nabla \zeta =-{m\zeta \nabla \varPhi }.\\
		\end{array} \right. $$
		Moreover, $F>0$ for all $t< t_0$ and $F(p_0,t_0)=0$~(the first vanishing point) imply
		$$\frac{\partial _-}{\partial t}F\left| _{\left( p_0,t_0 \right)} \right.\le 0.$$
		It follows that
		\begin{eqnarray*}
			0&\le& \varDelta G\\ 
			&=&-\varPhi ^m\varDelta \zeta -\zeta \varDelta \varPhi ^m-2\left< \nabla \varPhi ^m,\nabla \zeta \right> \\
			&=&-\varPhi ^m\varDelta \zeta -m\zeta \varPhi ^{m-1}\varDelta \varPhi -m\left( m-1 \right) \zeta \varPhi ^{m-2}\left| \nabla \varPhi \right|^2-2m\varPhi ^{m-1}\left< \nabla \varPhi ,\nabla \zeta \right>\\ 
			&\le& \varPhi ^m\left( -\frac{\partial _-}{\partial _t}\lambda +L(p_0,t_0)\lambda +\varepsilon \right) +m\lambda \varPhi ^{m-1}\left( -\frac{\partial _-}{\partial t}\varPhi \left( p_0,t_0 \right) \right) -2m\varPhi ^{m-1}\left< \nabla \varPhi ,\nabla \zeta \right> \\
			&=&\frac{\partial _-}{\partial t}F-\eta '+\varPhi _{}^{^m}\left( L(p_0,t_0)\lambda +\varepsilon \right) +2m^2\varPhi ^{m-2}\zeta \left| \nabla \varPhi \right|^2\\ 
			&\le& -\eta '+L(p_0,t_0)\eta +\varepsilon \varPhi _{}^{^m}+2m^2\eta \frac{\left| \nabla \varPhi \right|^2}{\varPhi ^2}.\\
		\end{eqnarray*}
		Let~$L_0=\max _{Q^{1,0}(x_0)\times \left[ 0,T \right]}L$,~
		$a_0=\max _{Q^{1,0}(x_0)\times \left[ 0,T \right]}|\lambda |$. At the point~$( p_0,t_0)$,~the following holds:
		$$
		\frac{1}{\varPhi _{}^{m}}=\frac{\lambda}{\eta}\le \min \left\{ \frac{a_0}{\eta \left( t_0 \right)},~C\frac{\sqrt{a}}{\sqrt{t_0}\eta \left( t_0 \right)} \right\} 
		.$$
		Taking the limit as~$\varepsilon \to 0$,~we obtain the following inequality at the point ~$( p_0,t_0)$~
		\begin{equation}
			\begin{split}
				\eta '\left( t_0 \right) &\le \eta \left( t_0 \right) \left( L(p_0,t_0) +2m^2\frac{|\nabla \varPhi |^2}{\varPhi _{}^{2}\left( t_0 \right)} \right) \\
				&\le \left\{ \begin{array}{c}
					\eta \left( t_0 \right) \left( L_0+C^{2}m^2\left( \frac{a_0}{\eta \left( t_0 \right)} \right) ^{\frac{2}{m}} \right) ,\ \text{or}\\
					\eta \left( t_0 \right) \left( at{_0^{-1}}+C^{2}m^2\left( \frac{a}{t_0\eta \left( t_0 \right)} \right) ^{\frac{1}{m}} \right),\\
				\end{array} \right. 
			\end{split}
			\label{lamda}
		\end{equation}
		where we continue to denote by $C$ a generic constant depending only on the dimension $n$ and the parameter $\alpha$.
		
		We begin by establishing estimates for ~$\lambda$. First, we prove that ~$\lambda \left( t \right) =O\left( t^{1/2} \right) $.~
		For any ~$1>{\delta }>0$,~let~${\eta }=t^{1/2}+\delta $. Then~$F>0$~near~$t=0$. Moreover, from the first inequality in~\eqref{lamda},~we have$$
		\frac{1}{2}t_0^{-\frac{1}{2}}\le \left( t_{0}^{\frac{1}{2}}+\delta \right) \left( L_0+\frac{C^{2}m^2a_{0}^{\frac{2}{m}}}{\left( t_{0}^{\frac{1}{2}}+\delta \right) ^{\frac{2}{m}}} \right) 
		.$$
		
		Let $m=2$. The preceding inequality implies that $t_0$ admits a uniform positive lower bound $\tau >0$ ( independent of $\delta$ ).~Consequently, as~$\delta \to 0$,~the following estimate holds for all $(x,t)$ with $X_t(x)\in 
		Q^{\sqrt{\frac{1}{2}-2nt},t}\left(x_0  \right) 
		$ and $t$ sufficiently small:~$$
		\lambda \le 2t^{\frac{1}{2}}
		.$$

		We proceed to derive refined estimates for $\lambda~(t \to 0^+)$. Let~$k\ge 1$~be a positive integer and~$\delta >0$, and define the auxiliary function~$
		\eta =\delta t^{\frac{1}{4}}+t^k
		$. From the first inequality in~\eqref{lamda},~we have  
		$$
		\frac{1}{4}\delta t_{0}^{-\frac{1}{4}}+kt_{0}^{k-1}\le \left( \delta t_{0}^{\frac{1}{4}}+t_{0}^{k} \right) \left( L_0+\frac{C^{2}m^2a_{0}^{\frac{2}{m}}}{\left( \delta t_{0}^{\frac{1}{4}}+t_{0}^{k} \right) ^{\frac{2}{m}}} \right) 
		.$$
		Select $m$ sufficiently large so that~$
		\frac{2k}{m}<1$. Then there exists ~$\tau_1 >0$ such that $t_0>\tau_1$.~Consequently in a neighborhood of $t=0$,~the estimate~$$
		\lambda \le 2t^{k}
		$$holds for all~$X_t(x)\in 
		Q^{\sqrt{\frac{1}{2}- 2nt},t}\left(x_0 \right) 
		$. 
		
		We now show that for each~$l\ge \alpha +1$,~there exists $\tau_2(n,\alpha,l)>0$such that $t_0>\tau_2$. Let~$
		\eta =\frac{1}{2}t^l
		$. $F>0$~near~$t=0$.~By the second inequality in~\eqref{lamda},~we can conclude that$$
		lt_{0}^{l-1}\le t_{0}^{l}\left( \frac{a}{t_0}+C^{2}m^2\left( \frac{a}{t_{0}^{l+1}} \right) ^{\frac{1}{m}} \right) .
		$$
		That is     $$
		t_{0}^{l-1}\le C^{2}m^2a ^{\frac{1}{m}}t_{0}^{l-\frac{1}{m}\left( l+1 \right)}
		.$$
		Fix $m$ sufficiently large such that~$
		\frac{1}{m}\left( l+1 \right) <\frac{1}{2}
		$. Then ~$
		t_0\ge \tau_2\left( n,\alpha ,l \right) 
		$,~and consequently, the estimate $$
		\lambda \left( x_0,t \right) \le t^l
		$$holds. Set~$l=\alpha +2$. By choosing a sufficiently small constant $S(a, \alpha)$, we obtain that for all $t < S(a, \alpha)$ the following inequality holds:
		\[
		\mathrm{II}(x_0, t) - \alpha \operatorname{H}(x_0, t)g(x_0,t) \geq -\lambda g(x_0,t) \geq -t^lg(x_0,t) \geq -g(x_0,t).
		\]
		
	\end{proof}

	\subsection{Time-decay estimate}  
	Using the estimate in Lemma~\ref{umblic} for mean curvature flow, we derive an $\frac{a}{t}$ estimate for the second fundamental form.
	
	\begin{lemma}\label{dey}
		For any $\alpha>0$,~suppose a non-compact hypersurface~$X:M^n\to \mathbb{R}^{n+1}$~evolving under mean curvature flow satisfies:
		
		(1)~$Q^{1,t}(x_0)\subset\subset \mathcal{M}_t$ for $t\in [0,T),$
		
		(2)~$\mathrm{II}\ge(\alpha\operatorname{H}-1)g$ in $Q^{1,t}(x_0)\times[0,T),$
		
		
		\noindent Then there exists ~$a(\alpha,n)>0$, $S_1(\alpha,n)>0$ such that
		\begin{eqnarray*} 
			|\mathrm{II}(x_0,t)|^2\leq at^{-1}   ~~~~ \text{for}  ~~~t\in (0,T\land S_1].
		\end{eqnarray*} 
	\end{lemma}
	\begin{proof}
		We argue by contradiction. If the conclusion of Lemma \ref{dey} fails, then there exists a sequence of non-compact hypersurface of mean curvature flows $X^{(k)}(t)$ defined for $M^{(k)} \times [0, T_k)$ and points $x_k \in M^{(k)}$ such that each $X^{(k)}(t)$ satisfies the assumptions of Lemma \ref{dey}, but for arbitrarily small $T_k > 0$ and arbitrarily large $a_k>0$, the inequality
		\[
		|\mathrm{II}_k(x_k, t)|^2 \leq a_k t^{-1}~~~~for\quad t \in [0, T_k),
		\]
		fails for every $k$.
		
		We can assume $a_k\to \infty$ and $a_kT_k\to 0$.~By the smoothness of each mean curvature flow,~one can select~$t_k\in(0,T_k]$~such that for each $\mathcal{M}^{(k)}_t=X^{(k)}_t(M^{(k)})$ with induced metric $g_k$, second fundamental form $\mathrm{II}_k$ and mean curvature $\mathrm{H}_k$,
		the following hold:

		(1)~$Q_k^{1,t}(x_k)\subset\subset \mathcal{M}^{(k)}_t$ for $t\in [0,T_k),$
		
		(2)~$\mathrm{II}_k\ge(\alpha\operatorname{H}_k-1)g_k$ in $ Q_k^{1,t}(x_k)\times[0,T_k),$
		
		
		(3)~$|\mathrm{II}_k(x_k,t)|^2<a_kt^{-1}$ for $ (0,t_k),$
		
		(4)~$|\mathrm{II}_k(x_k,t_k)|^2=a_kt_k^{-1}.$\\
		We denote by $  Q_k^{r, t}(p)\doteq \{x\in \mathcal{M}^{(k)}_t=X^{(k)}_t(M^{(k)}) \big|\left|  x  - X^{(k)}_t(p)  \right|<r\}$ for any $p\in M^{(k)}$.
		
		Applying Lemma~\ref{pp} to each flow~$X^{(k)}(t)$,~we can find a constant~$\beta>0$,~times~$\tilde{t_k}\in (0,t_k]$,~and points~$X^{(k)}_{\tilde{t}_k}(\tilde{x_k})\in Q^{1-\beta \sqrt{a_k\tilde{t}_k},{\tilde{t_k}}}(x_k) $~such that      
		\begin{eqnarray*}
			|\mathrm{II}_k(x,t)|^2\leq 4|\mathrm{II}_k(\tilde{x_k},\tilde{t_k})|^2=4A^2_k,
		\end{eqnarray*}
		where~$
		\left( X^{(k)}_t(x),t \right) \in Q^{\frac{1}{8}\beta a_kA_{k}^{-1},{ \tilde{t}_k }}\left( \tilde{x}_k\right) \times \left[ \tilde{t}_k-\frac{1}{8}a_kA_{k}^{-2},\tilde{t}_k \right] 
		$,~$\tilde{t}_kA_k^2\ge a_k\rightarrow +\infty $.
		
		At each point~$(\tilde{x_k},\tilde{t_k})$,~we translate $X^{(k)}(\tilde{x_k})$ to the origin and perform the  rescaling~$\tilde{X}^{(k)}(x,t)=A_k(X^{(k)}(x,\tilde{t_k}+A{^{-2}_k}t)-X^{(k)}(\tilde{x}_k,\tilde{t_k}))$ for $t\in [-\frac{1}{8}a_k,0]$.~Then
		
		
		(a)~$|\widetilde{\mathrm{II}}_k(\tilde{x_k},0)|^2=1,$
		
		(b)~$|\widetilde{\mathrm{II}}_k(x,t)|^2\leq4$ on $ \widetilde{Q}^{\frac{1}{8} \beta a_k,0}_k(\tilde{x_k}) \times [-\frac{1}{8}a_k,0],$
		
		(c)~$\widetilde{\mathrm{II}}_k\ge(\alpha\widetilde{\operatorname{H}}_k-A_k^{-1})\widetilde{g}_k$ on $ \widetilde{Q}^{\frac{1}{8} \beta a_k,0}_k(\tilde{x_k}) \times [-\frac{1}{8}a_k,0].$ \\
		By the compactness theorem for hypersurfaces (Theorem 11.6~in~\cite{11}),~we can obtain a subsequence ~$\tilde{X}^{(k_i)}$  converges smoothly on compact sets to an ancient solution~$X_{\infty}(t)$ for $t\in (-\infty,0]$ with
		$$
		~|\mathrm{II}_{\infty}|\le 4~~~\text{and}~~~\mathrm{II}_{\infty}\ge\alpha\operatorname{H}_{\infty}\tilde{g}_{\infty}\ge 0.
		$$
		
		Next, we show that the ancient solution~$X_{\infty}(t)$ must be of Type I, i.e. $\sup\limits_{t\in(-\infty,0]}|t||\mathrm{II}_{\infty}|<\infty$. Otherwise, suppose that $\sup\limits_{t\in(-\infty,0]}|t||\mathrm{II}_{\infty}|=\infty$. As Hamilton did in \cite{hamilton1995formation}, 
		we choose any sequence of times $T_i \to -\infty$ and positive numbers $\varepsilon_i \to 0$, and select a sequence of points $(x_i, t_i)$ such that
		\[
		|t_i|\,(t_i - T_i)\, \mathrm{H}_{\infty}^2(x_i, t_i) 
		\geq (1 - \varepsilon_i) \sup_{M_{\infty} \times [T_i, 0]} |t|\,(t - T_i)\, \mathrm{H}_{\infty}^2(x, t).
		\]
		Let $R_i \doteq \operatorname{H}^2_{\infty}\left(x_i, t_i\right)$ and define
		$
		\alpha_i \doteq\left(T_i-t_i\right) R_i \quad \text { and } \quad \omega_i \doteq-t_i R_i .
		$
		Then
		$$
		\begin{aligned}
			\frac{1}{-\alpha_i^{-1}+\omega_i^{-1}}=\frac{\left|t_i\right|\left(t_i-T_i\right) R_i}{\left|T_i\right|}&\geq (1 - \varepsilon_i) \left|T_i\right|^{-1}\sup_{M_{\infty} \times [T_i, 0]} |t|\,(t - T_i)\, \mathrm{H}_{\infty}^2(x, t)\\
			&\geq \frac{1 - \varepsilon_i}{2} \sup_{M_{\infty} \times [\frac{T_i}{2}, 0]} |t|\, \mathrm{H}_{\infty}^2(x, t) \to +\infty.
		\end{aligned}
		$$
		so that
		$
		-\lim _{i \rightarrow \infty} \alpha_i=\infty=\lim _{i \rightarrow \infty} \omega_i. 
		$
		Consider 
		$$
		\tilde{X}_{\infty}^{(i)}(x,t) = R_i^{\frac{1}{2}} \left( X_{\infty}(x, t_i + R_i^{-1} t) - X_{\infty}(x_i, t_i) \right),
		$$
		which is defined on the time interval $\left(-\infty, \omega_i\right]$, and whose mean curvature satisfies
		$$
		\left( \mathrm{H}^{(i)}_{\infty}(t) \right)^2 \leq \frac{\alpha_i \omega_i}{(1 - \epsilon_i)(\alpha_i - t)(\omega_i - t)} \quad \text{on } [\alpha_i, \omega_i].
		$$
		Then we have that $\tilde{X}_{\infty}^{(i)}$ subconverges to an $\alpha$-pinched eternal solution whose mean curvature is at most $1$ and attains the value $1$ at $t = 0$.  Thus, by the rigidity case of the Hamilton differential Harnack inequality\cite{Hamha}, this limit is a non-flat translator, which  contradicts  Proposition \ref{no_existence_te}. This proves  the ancient solution~$X_{\infty}(t)$ must be of Type I.

		Note that, by Proposition~5 and (10) of \cite{14}, the second fundamental form of the ancient solution $X_{\infty}(t)$ satisfies $|\mathrm{II}_{\infty}| \to 0$ uniformly in time as the spatial distance tends to infinity. 
		Since $X_{\infty}(t)$ is of Type I, we can apply Lemma~\ref{umblic} on $[T_i, 0]$ and, letting $T_i \to -\infty$, conclude that the ancient solution $X_{\infty}(t)$ is a totally umbilical hypersurface. Since $\mathrm{II}_{\infty}>0$ , we have $\operatorname{Ric}(\tilde{g}_{\infty}(0))>c>0$.~The lower bound on Ricci curvature implies an upper bound for the diameter, 
		which contradicts the noncompactness of the hypersurface. 
		
	\end{proof}
	
	\section{Proof of the Main Theorem}\label{section4}
	This section is devoted to the proofs of Theorem~\ref{lcmcf}~and~Theorem \ref{czx}.

	\textbf{The proof of Theorem~\ref{lcmcf}:}
	
	Before proceeding with the proof,~we first fix the necessary constants.~Given any pinching constant~$\alpha >0$ and dimension $n>2$, we define: 
	
	\begin{itemize}
		\item ~$\lambda(n) ,~\gamma(n)\text{~be the constants from Theorem~\ref{zhu}~},$
		\item $\beta(n)\text{~be the constant from Lemma~\ref{pp}~},$
		\item $a_1(\alpha,n)\text{~be the constant from Lemma~\ref{dey}~},$
		\item $a_0=\text{max}\{1,\lambda\gamma,\lambda(a_1+\gamma)\},$
		\item $S(a_0,\alpha)\text{~be the constant from Lemma~\ref{weakpc}~},$
		\item $\lambda_0=16\text{max}\{\beta,{(a_1a_0)}^{-\frac{1}{2}},S^{-\frac{1}{2}},S_1^{-\frac{1}{2}} \},$
		\item $\mu=\sqrt{1+\gamma {a_1}^{-1}}-1>0$.
	\end{itemize}

	Without loss of generality, we may assume that $X_0(p)$ is the origin in $\mathbb{R}^{n+1}$. We choose a sufficiently small positive constant~$\rho$~such that the second fundamental form satisfies~$| \mathrm{II}(x)|^2 \leq \rho^{-2}$ in $ Q^{R+4}(p)$.~Applying Lemma~\ref{wbh}~and Theorem~\ref{zhu}~to the region~$ Q^{R+4}(p)$,~we obtain a mean curvature flow~$X(t)$~on~$ Q^{R+3}(p)$~with initial condition~$X(0)=X_0~for~t\in [0 ,\gamma(n) \rho^2]$~and~$| \mathrm{II}(x,t)|^2 \leq \lambda(n) \rho^{-2}.$       
	From our choice of constants where~$a_0\geq \lambda\gamma$,~we can express the curvature condition as~$| \mathrm{II}(x,t)|^2 \leq a_0 t^{-1}.$  
	
	We define recursive sequences for the time~$t_k$~and radius~$r_k$~as follows:  
	
	(i)~$t_1=\gamma \rho^2,~r_1=R+3,$
	
	(ii)~$t_{k+1}=(1+\gamma {a_1}^{-1})t_k,~k\geq 1,$
	
	(iii)~$r_{k+1}=r_k-\lambda_0 \sqrt{a_0 t_k},~k\geq 1.$\\
	\textbf{Let} $\mathcal{P}(k)$ \textbf{be the following statement:~There exists a smooth mean curvature flow}~$X(t)$ \textbf{defined on}~$Q^{r_k,0}(p)\times [0,t_k]$~\textbf{satisfying initial condition}~$X(0)=X_0$~\textbf{with curvature bound}~$| \mathrm{II}(x,t)|^2 \leq a_0 t^{-1}.$
	
	~Note that by proper choice of~$\rho$,~$\mathcal{P}(1)$~holds. The goal is to prove~$\mathcal{P}(k)$~holds for all positive integers~$k$~when~$r_k\ge 0$.        
	Assume that $\mathcal{P}(k)$~holds.~Then there exists a mean curvature flow~$X(t)$~defined on~$Q^{r_k,0}(p)\times [0,t_k]$~that satisfies the curvature bound~$$| \mathrm{II}(x,t)|^2 \leq a_0 t^{-1}.$$We next prove $\mathcal{P}(k+1)$ holds.
	
	Fix a point ~$X_0(y)\in Q^{r_{k+1}+\frac{1}{2}\lambda _0\sqrt{a_0t_k},0  }\left( p \right) $.~Since $r_{k+1}=r_k-\lambda_0 \sqrt{a_0 t_k},$ the following holds:$$
	Q^{\frac{1}{4}\lambda _0\sqrt{a_0t_k},0 }\left( y \right) \subset \subset Q^{r_k,0 }\left(p\right). 
	$$
	Let~$\lambda _1=\frac{1}{4}\lambda _0\sqrt{a_0t_k}$.~We consider the rescaled mean curvature flow:~$
	\tilde{X}^{\lambda _{1}}\left(x, t \right) =\lambda _{1}^{-1}X\left(x, \lambda _{1}^{2}t \right) $ with second fundamental form $\mathrm{\tilde {II}}^{\lambda _1} $, mean curvature $\tilde {\operatorname{H}}^{\lambda _1} $ and induced metric $\tilde g^{\lambda _1} $ for~$t\in \left[0, 16\lambda _{0}^{-2}a_{0}^{-1} \right] 
	$ and $x \in \tilde{Q}_{\lambda _1}^{1,0  }\left(y \right)=
	Q^{\frac{1}{4}\lambda _0\sqrt{a_0t_k},0}\left(y \right),$ where  $  \tilde{Q}^{1,0  }_{\lambda _1}\left(y \right)\doteq B(\tilde{X}^{\lambda _{1}}_0(y),1)\cap \tilde{M}^{\lambda _{1}}_0 $ for $\tilde{M}^{\lambda _{1}}_t=\tilde{X}^{\lambda _{1}}_t(M).$
	
	Observe that both the pinching condition and the curvature decay estimate are invariant under this rescaling. Consequently, for $\tilde{X}^{\lambda _{1}}\left(x, t \right) $ the following hold  on $ \tilde{Q}^{1,0  }_{\lambda _1}\left(y \right)$: 
	~$$ \mathrm{\tilde {II}}^{\lambda _{1}}\left( x,0\right) \geq \alpha  \tilde {\operatorname{H}}^{\lambda _{1}}\left( x,0\right)\tilde g^{\lambda _{1}}(x,0)\ge 0~~~and~~~| \mathrm {\tilde {II}}^{\lambda _{1}}(t)|^2 \leq a_0 t^{-1}.$$
	Applying Lemma \ref{weakpc}~to $\tilde{X}^{\lambda _{1}} $,~with the parameter~$\lambda _{0}\geq 16S^{-\frac{1}{2}}$~and~$a_0\geq 1$,~we derive the curvature estimate: 
	$$ \tilde {\mathrm{II}}^{\lambda _{1}}\left(y,t\right) \geq (\alpha  \tilde {\operatorname{H}}^{\lambda _{1}}\left(y,t\right) -1)\tilde g^{\lambda _{1}}(y,t),~t\in [0 , 16\lambda _{0}^{-2}a_{0}^{-1} ].$$
	We obtain the following weak pinching estimate on the spacetime domain $
	Q^{r_{k+1}+\frac{1}{2}\lambda _0\sqrt{a_0t_k},0    }\left( p \right) \times \left[ 0,t_k \right] 
	$~:$$ \mathrm{II}\left(x,t\right) \geq (\alpha  \operatorname{H}\left(x,t\right)-\left( \frac{1}{4}\lambda _0\sqrt{a_0t_k} \right) ^{-1})g(t) 
	.$$
	
	For any point~$X_0(y)\in Q^{r_{k+1}+\frac{1}{4}\lambda _0\sqrt{a_0t_k},0 }\left(p\right)$, Lemma ~\ref{ssq} implies the following: 
	$$Q^{\frac{1}{8}\lambda _0\sqrt{a_0t_k},t  }\left( y \right) \subset Q^{\frac{1}{4}\lambda _0\sqrt{a_0t_k},0}\left( y\right) \subset \subset Q^{r_{k+1}+\frac{1}{2}\lambda _0\sqrt{a_0t_k},0 }\left(p \right) ~~\text{for}~~~t\in [0,t_k].$$
	Let~$ \lambda _{2} =\frac{1}{8}\lambda _0\sqrt{a_0t_k}$ and define the rescaled flow~$\tilde{X}^{\lambda _{2}}\left(x, t \right) =\lambda _{2}^{-1}X\left(x, \lambda _{2}^{2}t \right)$ for $ t\in \left[0, 64\lambda _{0}^{-2}a_{0}^{-1} \right]$ and $ x \in \tilde{Q}_{\lambda _{2}}^{1,t}\left(y \right)=Q^{\frac{1}{8}\lambda _0\sqrt{a_0t_k},{\lambda _{2}^{2}t } } \left( y \right) $. Then the mean curvature flow $\tilde{X}^{\lambda _{2}} $ preserves the weak pinching estimate on $\tilde{Q}_{\lambda _{2}}^{1,t}\left(y \right)$ :  
	$$ \tilde {\mathrm{II}}^{\lambda _{2}}\left( x,t\right) \geq (\alpha  \tilde {\operatorname{H}}^{\lambda _{2}}\left(x,t\right) - \frac{1}{2})\tilde g^{\lambda _{2}}(x,t)\geq (\alpha  \tilde {\operatorname{H}}^{\lambda _{2}}\left( x,t\right) -1)\tilde g^{\lambda _{2}}(x,t).$$ 
	Applying Lemma~\ref{dey}~to~$\tilde{X}^{\lambda _{2}}$~and using the scale invariance of curvature decay estimate, we obtain the following curvature decay estimate for the original flow $X(t)$:  
	\begin{equation}
		| \mathrm{II}(x, t)|^2 \leq \frac{a_1}{t},t \in (0, t_k].
		\label{5}
	\end{equation}

	
	
	
	Let ~$
	\varOmega =Q^{r_{k+1}+\frac{1}{8}\lambda _0\sqrt{a_0t_k},0  }\left(p\right) 
	$~and denote~$X_k=X(t_k)$.~The estimate~\eqref{5}~implies the following uniform curvature bound~$
	sup_{\varOmega}|\mathrm{II}|\le \rho ^{-2}=\left( \sqrt{a_{1}^{-1}t_k} \right) ^{-2}.$  
	Since by the parameter choices~$\lambda _0 \geq 16 \beta$~and~$\lambda _0\sqrt{a_0a_1} \geq 16$,~an application of Lemma~\ref{ssq} yields, for any~$
	X_k(x)\in Q^{r_{k+1},0   }\left(p\right) 
	$,~the following inclusions:  $$
	Q^{\rho ,{t_k}  }\left( x\right) \subset Q^{\frac{1}{16}\lambda _0\sqrt{a_0t_k},{t_k} }\left( x \right) \subset Q^{\frac{1}{8}\lambda _0\sqrt{a_0t_k},0   }\left(x\right) \subset \subset \varOmega 
	.~$$
	This implies ~$
	Q^{r_{k+1},0   }\left( p \right) \subset \varOmega _{\rho}
	$,~where~$\varOmega _{\rho}=\left\{ x\in \varOmega \ |\ Q^{\rho,{t_k} }\left( x \right) \subset \subset \varOmega \right\} $.~
	By applying Theorem~\ref{zhu}~we obtain a mean curvature flow~$X(t)$~defined on the spacetime region~$
	Q^{r_{k+1} ,0  }\left(p \right) \times \left[ t_k,t_k+\gamma \rho ^2 \right] 
	$.~This extends the flow forward in the time interval $[t_k,t_{k+1}]$.~By our initial assumptions on the constants,~which satisfies ~$
	\lambda \left( a_1+\gamma\right) \le a_0$~and~$t_k\left( 1+\gamma a_{1}^{-1} \right) =t_{k+1}
	$,~ we derive the following uniform curvature decay for the extended mean curvature flow~$X(t)$:~$$| \mathrm{II}\left( x,t) \right|^2\le \lambda \rho ^{-2}=\lambda a_1t_{k}^{-1}\le a_0t^{-1}.$$
	This verifies that~$\mathcal{P}(k+1)$~holds true. 
	
	Finally,~we establish a positive lower bound for the time $t_i$. Since $r_j\rightarrow -\infty  ~\text{as}~ j\rightarrow +\infty  $ and $r_1=R+3$, there exists an integer $i\in \mathrm{N}^+$  such that~$r_i\ge R+1$~and~$r_{i+1}<R+1.$~Indeed,~when~$r_i>0$~,~$P(i)$~holds.~Thus  
	\begin{equation*}
		\begin{split}
			R+1>r_{i+1}&=r_1-\lambda _0\sqrt{a_0}\sum_{k=1}^i{\sqrt{t_k}}\\
			&\ge R+3-\lambda _0\sqrt{a_0t_i}\sum_{k=0}^{\infty}{( 1+\mu) ^{-k}}\\
			&=R+3-\sqrt{t_i}\frac{\lambda _0\sqrt{a_0}\left( 1+\mu \right)}{\mu},
		\end{split}
		\label{eq:label}
	\end{equation*}
	it follows that $$
	t_i>\frac{4\mu^2}{a_0\lambda _{0}^{2}\left( 1+\mu \right) ^2}=T\left( \alpha, n\right) 
	.$$
	
	By our assumptions, the constant~$a_0$~depends only on the dimension~$n$~and the pinching constant~$\alpha$. Now we let constant $C =a_0 $.~In precise terms, given a fixed pinching constant~$\alpha$,~there exists a mean curvature flow solution~$X(t)$~defined on the spacetime domain~$
	Q^{R+1,0  }\left(p \right) \times \left[ 0,T \right] 
	$~with the second fundamental form decay~$| \mathrm{II}(x, t)|^2 \leq \frac{C}{t}.$
	For any ~$
	X_0(x_0)\in Q^{R,0  }\left( p \right) 
	$,~we apply Lemma~\ref{weakpc} again to the unit ball~$
	Q^{1,0 }\left( x_0 \right) 
	$~to obtain the weak pinching estimate $\mathrm{II}(x, t) \geq \left(\alpha  \operatorname{H}(x, t)-1\right)g(x,t)$,
	~thereby completing the proof of Theorem $\ref{lcmcf}$.\qed\\

	We now present the proof of our main theorem.\\  
	\textbf{The proof of Theorem~\ref{czx}:}
	Since the hypersurface $\mathcal{M}=X_0(M)$ is complete,~there exists an exhaustion~$\{\mathcal{M}^{(i)} \}$~satisfying:
	
	(a)~$\mathcal{M}^{(i)}=B_{i} \cap \mathcal{M} $ is an open subset of~$\mathcal{M}$ for each integer $i \geq 4$,
	
	(b)~$\bar{\mathcal{M}}^{(i)}\subset \mathcal{M}^{(i+1)}$,
	
	(c)~$\underset{i=1}{\overset{\infty}{\cup}}\mathcal{M}^{(i)}=\mathcal{M}$.
	
	For each~$\mathcal{M}^{(i)}$,~we define $\tilde{X}^{(i)}=\frac{1}{i}X$ and denote the rescaled hypersurfaces by $\tilde{\mathcal{M}}^{(i)}$. The rescaled hypersurfaces satisfy  $$\tilde{\mathrm{II}} \geq \alpha  \tilde{\operatorname{H}}\tilde{g}\geq 0.$$  
	
	By Theorem~\ref{lcmcf},~there exists a smooth mean curvature flow~$\tilde{X}^{(i)}(t)$~defined on~$B_1\times [0,T]$~with the following properties for all~$(x, t) \in B_1 \times[0, T]$~:   
	$$
	\left\{ \begin{array}{l}
		\tilde{X}^{(i)}=\tilde{X}^{(i)}(0),\\
		| \tilde{\mathrm{II}}(x, t)|^2 \leq \frac{C}{t} ,\\
		\tilde{\mathrm{II}}(x, t) \geq \left(\alpha  \tilde{\operatorname{H}}(x, t)-1\right)\tilde{g}(x,t),\\
	\end{array} \right. 
	$$
	where $C$ is a constant depending only on~$n$ and $\alpha$.
	Define the rescaled mean curvature flow by
	~$X^{(i)}(t)=i\tilde{X}^{(i)}(i^{-2}t)$.~This yields a solution~$X^{(i)}(t)$~defined on ~$
	B_i \times \left[ 0,T{i}^{2} \right] 
	$~with:
	$$
	\left\{ \begin{array}{l}
		{X}^{(i)}={X}^{(i)}(0),\\
		| {\mathrm{II}}(x, t)|^2 \leq \frac{C}{t} ,\\
		{\mathrm{II}}(x, t) \geq \left(\alpha {\operatorname{H}}(x, t)-{i}^{-2}\right)g(x,t),\\
	\end{array} \right. 
	$$
	Consequently,~we obtain a sequence of smooth mean curvature flows~$\{X^{(i)}(t) \}$~that satisfy a curvature decay condition independent of~$i$.~From the results of Ecker –Huisken \cite{10}, the curvature decay estimate~$\frac{C}{t}$~implies uniform bounds (independent of~$i$~)~for all covariant derivatives of the second fundamental form.~By a standard diagonal argument,~there exists a subsequence~$\{X^{k^{(i)}}(t) \}\subset \{X^{(i)}(t) \}$~ that converges smoothly to a mean curvature flow~$X^{(\infty)}(t)$.~Since~$i\to \infty~(k^{(i)} \to \infty)$,~the limit flow~$X^{(\infty)}(t)$~satisfies the following properties for all~$(x,t)\in M\times [0,\infty)$:~       
	$$
	\left\{ \begin{array}{l}
		{X}={X}^{(\infty)}(0),\\
		| {\mathrm{II}}(x, t)|^2 \leq \frac{C}{t} ,\\
		{\mathrm{II}}(x, t) \geq \left(\alpha {\operatorname{H}}(x, t)\right)g(x,t).\\
	\end{array} \right. 
	$$
	This completes the proof of Theorem~\ref{czx}~.\qed

	Theorem~\ref{czx}, in combination with the result of Bourni--Langford--Lynch (Theorem~7 in~\cite{14}), yields a proof of Hamilton's extrinsic pinching theorem via the mean curvature flow. 
	However, since we have already shown that any mean curvature flow satisfying the hypotheses of Theorem~\ref{czx} must be of Type~III, we may restrict the argument of Bourni--Langford--Lynch to the Type~III case.  For the sake of completeness, we present the full argument below. 
	\begin{theorem}\cite{8} \label{hamilton1}
		Any smooth, proper, locally, uniformly convex hypersurface $\operatorname{X}:M^n \to \mathbb{R}^{n+1}$, $n \geq 2$, with a pinched second fundamental form:
		$$\mathrm{II}\geq \alpha \operatorname{H}g \text{ for some } \alpha>0,$$ is necessarily compact. 
	\end{theorem}
	\begin{proof}[Proof of Theorem   \ref{hamilton1}  via the mean curvature flow]
		We argue by contradiction. Suppose that $\operatorname{X}$ is noncompact. By Theorem~\ref{czx}, there exists a complete solution~$X(t)$ to the mean curvature flow~starting from $\operatorname{X}$ such that
		
		(1) the solution is of Type III, i.e. it exists for $t\in [0,+\infty)$ satisfying ~$| \mathrm{II}(x, t)|^2 \leq \frac{C}{t},$
		
		(2)~${\mathrm{II}}(x, t) \geq \alpha  \operatorname{H}(x, t)g(x,t)\geq0,$\\
		for all~$(x, t) \in M^n \times(0, +\infty)$.
		
		Pick a sequence of times $t_j>0$ with $t_j \rightarrow +\infty$. Consider the rescaled flow $\{X^{(j)}(t)\}_{t \in\left(-\lambda_j^2 t_j, \infty\right)}$  defined by $X^{(j)}(\cdot,t) \doteqdot \lambda_j X(\lambda_j^{-2} t+t_j)$, where $\lambda_j \doteqdot \frac{1}{\sqrt{t_j}}$. Since $\lambda_j^2 t_j \equiv 1$ and
		
		$$
		\operatorname{H}_j(\cdot, t)=\lambda_j^{-1} \operatorname{H}\left(\cdot, \lambda_j^{-2} t+t_j\right) \leq \frac{\Lambda}{\sqrt{t+1}},
		$$
		we obtain a subsequential limit defined for $t \in(-1, \infty)$. Furthermore,
		
		$$
		\sqrt{t+1} \operatorname{H}_{\infty}(\cdot, t)=\lim _{j \rightarrow \infty} \sqrt{t+1} \operatorname{H}_j(\cdot, t)=\lim _{j \rightarrow \infty} \sqrt{\frac{t+1}{\lambda_j^2}}\operatorname{H}\left(\cdot, \frac{t+1}{\lambda_j^2}\right) .
		$$

		By the Hamilton's differential Harnack inequality  \cite{Hamha} and the type-III hypothesis, the limit on the right exists (and is positive) independently of $t$. Thus, by the rigidity case of the differential Harnack inequality, the limit is a non-flat expander, which  contradicts  Proposition \ref{no_existence_te}.
	\end{proof}


\begin{thebibliography}{99}
		\bibitem{11}
		Andrews B, Chow B, Guenther C, et al. Extrinsic Geometric Flows. American Mathematical Society, 2022.
		
		\bibitem{14}
		Bourni T, Langford M, Lynch S. Pinched hypersurfaces are compact. Advanced Nonlinear Studies. 2023 Mar 22;23(1):20220046.
		
		\bibitem{26}
		Brendle S, Schoen R. Sphere theorems in geometry. Surveys in Differential Geometry, 2008, 13(1): 49-84.
		
		
		
		
		
		
		\bibitem{18}
		Chen B L, Zhu X P. Complete Riemannian manifolds with pointwise pinched curvature. Inventiones Mathematicae, 2000, 140: 423-452.
		
		\bibitem{21}
		Cheng L. On locally conformally flat manifolds with positive pinched Ricci curvature. To Appear Calculus of Variations and Partial Differential Equations, 2023.
		
		
		
		\bibitem{20}
		Deruelle A, Schulze F, Simon M. Initial stability estimates for Ricci flow and three dimensional Ricci-pinched manifolds. Duke Math. J. 174(17): 3433-3492
		
		\bibitem{10}
		Ecker K, Huisken G. Interior estimates for hypersurfaces moving by mean-curvarture . Inventiones Mathematicae, 1991, 105(1): 547-569. 
		
		
		\bibitem{4}
		Hamilton R S. Three-manifolds with positive Ricci curvature. Journal of Differential Geometry, 1982, 17(2): 255-306.
		
		\bibitem{8}
		Hamilton R S. Convex hypersurfaces with pinched second fundamental form. Communications in Analysis and Geometry, 1994, 2(1): 167-172.
		
		\bibitem{Hamha}
		Hamilton R S.  Harnack estimate for the mean curvature flow, Journal of Differential Geometry, Vol. 41, No. 1 (1995), pp. 215–226.
		
		\bibitem{hamilton1995formation}
		Hamilton R S.  
		\textit{The formation of singularities in the Ricci flow}. 
		In: {\it Surveys in Differential Geometry}, (1995), Vol. 2, 7–136.
		
		\bibitem{15}
		Huisken G. Flow by mean curvature of convex surfaces into spheres. Journal of Differential Geometry, 1984, 20(1): 237-266.
		
		\bibitem{22}
		Hochard R. Short-time existence of the Ricci flow on complete, non-collapsed $3 $-manifolds with Ricci curvature bounded from below. arXiv preprint arXiv:1603.08726, 2016.
		
		
		
		\bibitem{31} 
		Langford, M.  Local convexity estimates for mean curvature flow. Journal für die reine und angewandte Mathematik (Crelles Journal), 2023(800), 45-53.
		
		\bibitem{29}
		Lee, Man-Chun and Topping M Peter. Three-manifolds with non-negatively pinched Ricci curvature. 131 (3) 633-651, November 2025.
		
		\bibitem{LT}
		
		Lee, Man-Chun and Topping M Peter. Manifolds with PIC1 pinched curvature. Geometry \& Topology 29:9 (2025) 4767–4798
		
		
		\bibitem{28}
		Ma L, Cheng L. Yamabe flow and Myers type theorem on complete manifolds. Journal of Geometric Analysis, 2014, 24(1): 246-270.
		
		\bibitem{Ni}
		Ni L. Is a complete Riemannian manifold with positively pinched Ricci curvature compact, arXiv:2510.10279v2
		
		\bibitem{25}
		Ni L, Wu B. Complete manifolds with nonnegative curvature operator. Proceedings of the American Mathematical Society, 2007, 135(9): 3021-3028.
		
		
		
		\bibitem{19}
		Lott J. On 3-manifolds with pointwise pinched nonnegative Ricci curvature. Mathematische Annalen, 2023: 1-20.
		
		
		\bibitem{9}
		Zhu X P. Lectures on Mean Curvature Flows. American Mathematical Society, 2002.
		
	\end{thebibliography}
\end{document}